\documentclass[12pt,a4paper]{amsart}
\usepackage[utf8]{inputenc}
\usepackage[english]{babel}
\usepackage{amsmath}
\usepackage{amsfonts}
\usepackage{amssymb}
\usepackage{graphicx}
\usepackage[a4paper]{geometry}

\usepackage{dsfont} 
\usepackage{centernot} 

\newcommand{\nwc}{\newcommand}
\nwc{\nwt}{\newtheorem}
\nwt{theo}{Theorem}
\nwt{coro}[theo]{Corollary}
\nwt{ex}[theo]{Example}
\nwt{prop}[theo]{Proposition}
\nwt{defin}[theo]{Definition}
\nwt{rem}[theo]{Remark}
\nwt{lem}[theo]{Lemma}

\nwc{\mf}{\mathbf} 
\nwc{\blds}{\boldsymbol} 
\nwc{\ml}{\mathcal} 

\nwc{\lam}{\lambda}
\nwc{\del}{\delta}
\nwc{\Del}{\Delta}
\nwc{\Lam}{\Lambda}
\nwc{\elll}{\ell}

\nwc{\IA}{\mathbb{A}} 
\nwc{\IB}{\mathbb{B}} 
\nwc{\IC}{\mathbb{C}} 
\nwc{\ID}{\mathbb{D}} 
\nwc{\IE}{\mathbb{E}} 
\nwc{\IF}{\mathbb{F}} 
\nwc{\IG}{\mathbb{G}} 
\nwc{\IH}{\mathbb{H}} 
\nwc{\IL}{\mathbb{L}}
\nwc{\IN}{\mathbb{N}} 
\nwc{\IP}{\mathbb{P}} 
\nwc{\IQ}{\mathbb{Q}} 
\nwc{\IR}{\mathbb{R}} 
\nwc{\IS}{\mathbb{S}} 
\nwc{\IT}{\mathbb{T}} 
\nwc{\IZ}{\mathbb{Z}} 

\def\bbleft{{\mathchoice {[\mskip-3mu {[}} {[\mskip-3mu {[}}{[\mskip-4mu {[}}{[\mskip-5mu {[}}}}
\def\bbright{{\mathchoice {]\mskip-3mu {]}} {]\mskip-3mu {]}}{]\mskip-4mu {]}}{]\mskip-5mu {]}}}}
\nwc{\setK}{\bbleft 1,K \bbright}
\nwc{\setN}{\bbleft 1,\cN \bbright}
 
\nwc{\va}{{\bf a}}
\nwc{\vb}{{\bf b}}
\nwc{\vc}{{\bf c}}
\nwc{\vd}{{\bf d}}
\nwc{\ve}{{\bf e}}
\nwc{\vf}{{\bf f}}
\nwc{\vg}{{\bf g}}
\nwc{\vh}{{\bf h}}
\nwc{\vi}{{\bf i}}
\nwc{\vI}{{\bf I}}
\nwc{\vj}{{\bf j}}
\nwc{\vk}{{\bf k}}
\nwc{\vl}{{\bf l}}
\nwc{\vm}{{\bf m}}
\nwc{\vM}{{\bf M}}
\nwc{\vn}{{\bf n}}
\nwc{\vo}{{\it o}}
\nwc{\vp}{{\bf p}}
\nwc{\vq}{{\bf q}}
\nwc{\vr}{{\bf r}}
\nwc{\vs}{{\bf s}}
\nwc{\vt}{{\bf t}}
\nwc{\vu}{{\bf u}}
\nwc{\vv}{{\bf v}}
\nwc{\vw}{{\bf w}}
\nwc{\vx}{{\bf x}}
\nwc{\vy}{{\bf y}}
\nwc{\vz}{{\bf z}}
\nwc{\bal}{\blds{\alpha}}
\nwc{\bep}{\blds{\epsilon}}
\nwc{\barbep}{\overline{\blds{\epsilon}}}
\nwc{\bnu}{\blds{\nu}}
\nwc{\bmu}{\blds{\mu}}
\nwc{\bet}{\blds{\eta}}

\nwc{\bk}{\blds{k}}
\nwc{\bm}{\blds{m}}
\nwc{\bM}{\blds{M}}
\nwc{\bp}{\blds{p}}
\nwc{\bq}{\blds{q}}
\nwc{\bn}{\blds{n}}
\nwc{\bv}{\blds{v}}
\nwc{\bw}{\blds{w}}
\nwc{\bx}{\blds{x}}
\nwc{\bxi}{\blds{\xi}}
\nwc{\by}{\blds{y}}
\nwc{\bz}{\blds{z}}

\nwc{\cA}{\ml{A}}
\nwc{\cB}{\ml{B}}
\nwc{\cC}{\ml{C}}
\nwc{\cD}{\ml{D}}
\nwc{\cE}{\ml{E}}
\nwc{\cF}{\ml{F}}
\nwc{\cG}{\ml{G}}
\nwc{\cH}{\ml{H}}
\nwc{\cI}{\ml{I}}
\nwc{\cJ}{\ml{J}}
\nwc{\cK}{\ml{K}}
\nwc{\cL}{\ml{L}}
\nwc{\cM}{\ml{M}}
\nwc{\cN}{\ml{N}}
\nwc{\cO}{\ml{O}}
\nwc{\cP}{\ml{P}}
\nwc{\cQ}{\ml{Q}}
\nwc{\cR}{\ml{R}}
\nwc{\cS}{\ml{S}}
\nwc{\cT}{\ml{T}}
\nwc{\cU}{\ml{U}}
\nwc{\cV}{\ml{V}}
\nwc{\cW}{\ml{W}}
\nwc{\cX}{\ml{X}}
\nwc{\cY}{\ml{Y}}
\nwc{\cZ}{\ml{Z}}

\nwc{\fA}{\mathfrak{a}}
\nwc{\fB}{\mathfrak{b}}
\nwc{\fC}{\mathfrak{c}}
\nwc{\fD}{\mathfrak{d}}
\nwc{\fE}{\mathfrak{e}}
\nwc{\fF}{\mathfrak{f}}
\nwc{\fG}{\mathfrak{g}}
\nwc{\fH}{\mathfrak{h}}
\nwc{\fI}{\mathfrak{i}}
\nwc{\fJ}{\mathfrak{j}}
\nwc{\fK}{\mathfrak{k}}
\nwc{\fL}{\mathfrak{l}}
\nwc{\fM}{\mathfrak{m}}
\nwc{\fN}{\mathfrak{n}}
\nwc{\fO}{\mathfrak{o}}
\nwc{\fP}{\mathfrak{p}}
\nwc{\fQ}{\mathfrak{q}}
\nwc{\fR}{\mathfrak{r}}
\nwc{\fS}{\mathfrak{s}}
\nwc{\fT}{\mathfrak{t}}
\nwc{\fU}{\mathfrak{u}}
\nwc{\fV}{\mathfrak{v}}
\nwc{\fW}{\mathfrak{w}}
\nwc{\fX}{\mathfrak{x}}
\nwc{\fY}{\mathfrak{y}}
\nwc{\fZ}{\mathfrak{z}}

\nwc{\tA}{\widetilde{A}}
\nwc{\tB}{\widetilde{B}}
\nwc{\tE}{E^{\vareps}}
\nwc{\tk}{\tilde k}
\nwc{\tN}{\tilde N}
\nwc{\tP}{\widetilde{P}}
\nwc{\tQ}{\widetilde{Q}}
\nwc{\tR}{\widetilde{R}}
\nwc{\tV}{\widetilde{V}}
\nwc{\tW}{\widetilde{W}}
\nwc{\ty}{\tilde y}
\nwc{\teta}{\tilde \eta}
\nwc{\tdelta}{\tilde \delta}
\nwc{\tlambda}{\tilde \lambda}
\nwc{\ttheta}{\tilde \theta}
\nwc{\tvartheta}{\tilde \vartheta}
\nwc{\tPhi}{\widetilde \Phi}
\nwc{\tpsi}{\tilde \psi}
\nwc{\tmu}{\tilde \mu}

\nwc{\To}{\longrightarrow} 

\nwc{\ad}{\rm ad}
\nwc{\eps}{\epsilon}
\nwc{\ep}{\epsilon}
\nwc{\vareps}{\varepsilon}

\def\N{\mathbb{N}}

\def\R{\mathbb{R}}

\nwc{\lap}{\bigtriangleup}
\nwc{\rest}{\restriction}
\nwc{\Diff}{\operatorname{Diff}}
\nwc{\diam}{\operatorname{diam}}
\nwc{\Res}{\operatorname{Res}}
\nwc{\Spec}{\operatorname{Sp}}
\nwc{\Vol}{\operatorname{Vol}}
\nwc{\Op}{\operatorname{Op}}
\nwc{\supp}{\operatorname{supp}}
\nwc{\Span}{\operatorname{span}}
\nwc{\weaksc}{\overset{\ast}{\rightharpoonup}}
\nwc{\wto}{\rightharpoonup}
\nwc{\charf}{\mathds{1}}
\nwc{\notimplies}{\centernot\implies}

\def\ep{\epsilon}

\def\sq2{\sqrt{2}}

\def\t2{{\mathbb T}^2}
\def\s2{{\mathbb S}^2}

\nwc{\dia}{\varepsilon}
\nwc{\cut}{f}
\nwc{\qm}{u_\hbar}

\def\hto0{\xrightarrow{\hbar\to 0}}

\def\rto0{\xrightarrow{r\to 0}}

\def\Tend#1#2{\mathop{\longrightarrow}\limits_{#1\rightarrow#2}}

\nwc{\la}{\langle}
\nwc{\ra}{\rangle}
\nwc{\lp}{\left(}
\nwc{\rp}{\right)}

\nwc{\bequ}{\begin{equation}}
\nwc{\be}{\begin{equation}}
\nwc{\ben}{\begin{equation*}}
\nwc{\bea}{\begin{eqnarray}}
\nwc{\bean}{\begin{eqnarray*}}
\nwc{\bit}{\begin{itemize}}
\nwc{\bver}{\begin{verbatim}}

\nwc{\eequ}{\end{equation}}
\nwc{\ee}{\end{equation}}
\nwc{\een}{\end{equation*}}
\nwc{\eea}{\end{eqnarray}}
\nwc{\eean}{\end{eqnarray*}}
\nwc{\eit}{\end{itemize}}
\nwc{\ever}{\end{verbatim}}

\title[Observabilty and fractional Schrödinger operators]{Observability results related to fractional Schrödinger operators}
\author[Fabricio Macià]{Fabricio Macià}

\address{M$^2$ASAI, Universidad Politécnica de Madrid.  \newline
ETSI Navales, Avda. de la Memoria, 4. 28040 Madrid, SPAIN.}

\email{fabricio.macia@upm.es}

\thanks{The author acknowledges the support of Ministerio de Ciencia, Innovación y Universidades of the Spanish goverment through grant MTM2017-85934-C3-3-P}

\begin{document}

\begin{abstract}
We establish observability inequalities for various problems involving fractional Schrödinger operators $(-\Delta)^{\alpha/2}+V$, $\alpha>0$, on a compact Riemannian manifold. Observability from an open set for the corresponding fractional Schrödinger evolution equation with $\alpha>1$ is proved to hold as soon as the observation set satisfies the Geometric Control Condition; it is also shown that this condition is necessary when the manifold is the $d$-dimensional sphere equipped with the standard metric. This is in stark contrast with the case of eigenfunctions. We construct potentials on the two-sphere with the property that there exist two points on the sphere such that eigenfunctions of $-\Delta+V$  are uniformly observable from an arbitrarily small neighborhood of those two points. This condition is much weaker than the Geometric Control Condition, which is necessary for uniform observability of eigenfunctions for the free Laplacian on the sphere. The same result also holds for eigenfunctions of  $(-\Delta)^{\alpha/2}+V$, for any $\alpha>0$.
\end{abstract}

\maketitle

\section{Introduction}
Consider a compact Riemannian manifold $(M,g)$, an open subset $\omega\subseteq M$ and the \textit{Laplace-Beltrami} operator (or simply the Laplacian) $\Delta$ on $M$ associated to the metric $g$. Since $-\Delta$ is a non-negative self-adjoint operator on $L^2(M)$, it is possible to define the fractional Laplacian $(-\Delta)^{\alpha/2}$ for $\alpha>0$ using the spectral theorem. 

In this article we are interested in understanding the validity of various observability estimates associated to $(-\Delta)^{\alpha/2}$ or, more generally, to the fractional Schrödinger operator $(-\Delta)^{\alpha/2}+V$, where $V$ is a bounded real potential.  

We first consider the fractional Schrödinger evolution equation:
\begin{equation}\label{e:fschro}
\left\{
\begin{array}{ll}
i\partial_t u(t,x)=(-\Delta )^{\alpha/2}u(t,x)+V(x)u(t,x),& (t,x)\in\IR\times M,\medskip\\ 
u|_{t=0}=u^0\in L^2(M),&
\end{array}
\right.
\end{equation}
where $\alpha>0$ and $V\in\cC^\infty(M;\IR)$. 

The fractional Schrödinger evolution is said to be observable from an open set $\omega\subseteq M$ and time $T>0$ provided a constant $C=C_{T,\omega}>0$ exists such that for every initial datum $u^0\in L^2(M)$ the corresponding solution to \eqref{e:fschro} satisfies the \textit{observability estimate}
\begin{equation}\label{e:obss}\tag{$O(T,\omega)$}
||u^0||_{L^2(M)}^2\leq C \int_0^T\int_\omega |e^{-it((-\Delta )^{\alpha/2}+V)}u^0(x)|^2dx\,dt.
\end{equation}
The observability property is crucial in establishing controllability results for the fractional Schrödinger evolution or uniform stabilization for the  semi-groups generated by bounded, non self-adjoint perturbations of $(-\Delta )^{\alpha/2}$, see \cite{LionsBook}.

We are also interested in understanding the validity of a weaker observability property, regarding the eigenfunctions of the fractional Schrödinger operator. Since $M$ is compact, $(-\Delta)^{\alpha/2}+V$ has compact resolvent and therefore its spectrum is discrete and bounded below,  and there exists an orthonormal basis  of $L^2(M)$ consisting of eigenfunctions of $(-\Delta)^{\alpha/2}+V$. 

We say that the eigenfunctions of the fractional Schrödinger operator  are observable from an open subset $\omega\subseteq M$ if a constant $C=C_{\omega}>0$ exists such that 
\begin{equation}\label{e:obse}\tag{$O_\mathrm{E}(\omega)$}
||\varphi||_{L^2(M)}\leq C_\omega \|\varphi\|_{L^2(\omega)},
\end{equation}
holds for every $L^2(M)$-eigenfunction of the Schrödinger operator:
\[
((-\Delta)^{\alpha/2}+V(x))\varphi(x)=\lambda\varphi(x),\quad x\in M.
\]
Note that it is crucial in this definition that the constant $C_\omega$ is required to be uniform with respect to the eigenvalue $\lambda$.

If $\varphi$ is an eigenfunction of $(-\Delta)^{\alpha/2}+V$ with $\|\varphi\|_{L^2(M)}=1$ and eigenvalue $\lambda$ then:
$u(t,\cdot)=e^{-it\lambda}\varphi$  is a solution to the fractional Schrödinger equation \eqref{e:fschro}.
Therefore \eqref{e:obss} implies \eqref{e:obse}. The observability property for eigenfunctions is interesting as it describes localization properties of high-energy eigenstates, a fundamental question in mathematical quantum mechanics. It has been shown in \cite{PTZ} that this property is related to controllability properties of wave equations with random initial data.

It is known since the seventies that both \eqref{e:obss} and \eqref{e:obse} hold under a geometric condition that relates the geometry of the observability region $\omega$ and the geodesics of the manifold $(M,g)$; more detailed references will be given after the statement of our first result. 

Let us first recall the precise definitions. The open subset $\omega\subseteq M$ satisfies the \textit{Geometric Control Condition (GCC)} provided that every geodesic of $(M,g)$ intersects $\omega$. Given $T>0$, we say that $\omega$ satisfies GCC$_T$ whenever every geodesic segment of length smaller than $T$ intersects $\omega$.
Since $M$ is compact, $\omega$ satisfies GCC if and only if it satisfies GCC$_T$ for some $T>0$.

\begin{theo}\label{t:suff}
Let $(M,g)$ be a compact smooth Riemannian manifold without boundary, $T>0$ and $\omega\subseteq M$ open. Then the following results hold.
\begin{itemize}
\item if $0<\alpha<1$ and $\overline{\omega}\neq M$ then \eqref{e:obss} does not hold no matter how large $T>0$ is.
\item if $\alpha=1$ and $\omega$ satisfes GCC$_T$ then \eqref{e:obss} holds. Conversely, if $M\setminus\overline{\omega}$ contains a closed geodesic segment of length $T$ then \eqref{e:obss} does not hold.
\item if $\alpha>1$ and $\omega$ satisfes GCC then \eqref{e:obss} holds for every $T>0$. 
\end{itemize}
\end{theo}
We give a unified proof of this result in Section \ref{s:suf}. Some of the statements in Theorem \ref{t:suff} hold for time-dependent potentials; Theorem \ref{t:tdep} in Section \ref{s:suf} presents their precise formulation. 

The case $\alpha=1$ corresponds to the wave equation and was first proved by Rauch and Taylor in \cite{RauchTaylor75}. The necessary condition is related to the existence of highly localized solutions to the wave equation known as Gaussian beams or wave-packets, see \cite{RalstonGB, MZ02}. Note that the necessary condition for observability is slightly weaker than GCC$_T$, since geodesics are just required to intersect the larger set $\overline{\omega}$; the sufficiency of this less restrictive condition depends in general of the particular configuration of geodesics that intersect $\overline{\omega}$ precisely at the boundary $\partial\omega$ (the so-called \textit{grazing rays}). 

When $\alpha=2$, problem \eqref{e:fschro} is the usual Schrödinger equation. The sufficiency of GCC in this case was established by Lebeau \cite{Lebeau1992}, in the more general setting of manifolds with boundary. We present here a rather direct proof in the boundaryless setting that remains valid for every $\alpha>1$.

Finally, note that when $\alpha>2$ (or $\alpha>1$ but $T$ large enough so that GCC$_T$ holds) and $V=0$ , the sufficiency of GCC can also be deduced from that of the case $\alpha=2$ from abstract, functional-theoretic arguments, see \cite[Theorem 3.5]{Mi12}.

In general, GCC is far from being necessary for the observability property when $\alpha=2$.  The simplest setting where  this kind of behavior takes place is the flat torus $\IT^d=\IR^d/\IZ^d$, where \eqref{e:obss} holds for every open set $\omega\subseteq \IT^d$. This was first proved by Jaffard \cite{Jaffard1990} for the free Laplacian when $d=2$, and generalized to the multidimensional case by Komornik \cite{Komornik1992}. The addition of a potential does not change the final result: see \cite{BZ12, BBZ14} for results on the two-dimensional case $d=2$ and \cite{AM14} for the general case. The article \cite{AFM15} provides more general results that encompass the general case $\alpha>1$ on the torus $\IT^d$.

Other situations where observability for the Schrödinger equation has been established include Zoll manifolds (\textit{i.e.} manifolds all of whose geodesics are closed) \cite{MaciaDispersion, MacRiv16} and the Euclidean disk \cite{ALM16,ALMCras}. On the disk, a necessary and sufficient condition for observability is that the observability region intersects the boundary of the disk in an open arc. Note that this situation is intermediate between the torus (for which observability holds unconditionally) and the sphere, for which it turns out that GCC is necessary. This is the content of the next result.
 
\begin{theo}\label{t:nes}
Suppose $(M,g)=(\IS^d,\mathrm{can})$ is the sphere equipped with its canonical metric. Let $\omega\subseteq \IS^d$ be an open set such that $M\setminus \overline{\omega}$ contains a geodesic. Then, for every $\alpha>0$ the following results hold.
\begin{itemize}
\item The observability estimate \eqref{e:obss} does not hold for any $T>0$;
\item If in addition $V=0$, then \eqref{e:obse} fails as well. 
\end{itemize}
\end{theo}
The proof of this theorem is given in Section \ref{s:nec}, and covers the case of time-dependent potentials for the first part of the statement. It is a variation of the results in  \cite{MaciaDispersion, MacRiv16} for $\alpha=2$, and here we follow closely the strategy of the proofs in those references. The statement concerning the observability of eigenfunctions for the free Laplacian on the sphere is well-known, although we give here a complete proof for the sake of completeness. 

It turns out that the conditions under which \eqref{e:obse} holds change dramatically when a perturbation $V$ is present. This was first analyzed in \cite{MacRiv19}, where it is shown that, on the $d$-dimensional sphere,  \eqref{e:obse} holds provided $\omega\subseteq \IS^d$ satisfies a geometric condition involving the potential $V$ (what we call here $V$-GCC, see Section \ref{s:vgcc}) that is in general much weaker than GCC. Here we present a proof for fractional Schrödinger operators based on an observability estimate over long times for the fractional Schrödinger evolution, see Theorem \ref{t:lt} in Section \ref{s:vgcc} which is of independent interest. This is based on the detailed study of the long-time dynamics of Schrödinger flows presented in \cite{MaciaAv, MacRiv16}.

Our last result, valid in two dimensions, gives an explicit construction of a family of smooth potentials on the sphere such that \eqref{e:obse} holds when $\omega$ is an arbitrarily small neighborhood of two fixed points on the sphere. This is in stark contrast with the same situation for the free Laplacian, for which GCC is a necessary condition for \eqref{e:obse}. 
\begin{theo}\label{t:minimal}
Suppose that $(M,g)=(\IS^2,\mathrm{can})$ and $\alpha>0$.There exist a non-empty family of potentials $\cT\subset\cC^\infty(\IS^2;\IR)$ such that, for every $V\in\cT$ there exist two distinct points $p,q\in\IS^2$ such that \eqref{e:obse} holds for any open set $\omega\subseteq \IS^2$ that contains $p,q$.
\end{theo}
This result is proved in Section \ref{s:potential}. As it will be clear from the proof, the size of $V$ in any reasonable norm can be arbitrarily small. The result follows from the analysis of the flow of a certain Hamiltonian vector field on the sphere constructed from the potential $V$.

\section{Observability for the fractional Schrödinger evolution on a compact manifold. Proof of Theorem \ref{t:suff}}\label{s:suf}
Most of the proofs we present remain valid when the potentials are time-dependent. In this section we will consider the fractional Schrödinger equation:
\begin{equation}\label{e:fracschro}
\left\{
\begin{array}{ll}
i\partial_t u(t,x)=(-\Delta )^{\alpha/2}u(t,x)+V(t,x)u(t,x),& (t,x)\in\IR\times M,\medskip\\ 
u|_{t=0}=u^0\in L^2(M),&
\end{array}
\right.
\end{equation}
where $\alpha>0$ and $V\in\cC^\infty_c(\IR\times M;\IR)$ is bounded together with all its derivatives. 
The following observability results still hold in this more general context.
\begin{theo}\label{t:tdep}
Let $(M,g)$ be a smooth compact Riemannian manifold without boundary and $\omega\subseteq M$ an open set. Suppose that the observability estimate 
\begin{equation}\label{e:obst}
||u^0||_{L^2(M)}^2\leq C \int_0^T\int_\omega |u(t,x)|^2dx\,dt.
\end{equation}
holds for every solution $u$ to \eqref{e:fracschro}. Then,
\begin{itemize}
\item when $0<\alpha<1$ the set $\omega$ must be dense in $M$.
\item when $\alpha=1$ the set $\overline{\omega}$ must intersect all closed geodesic segments of length $T$.
\end{itemize}
\end{theo}
The proof of this result, together with that of Theorem \ref{t:suff}, are presented at the end of this section. Its starting point is a  compactness argument based on the analysis of semiclassical defect measures. This type of approach goes back to \cite{Lebeau1996}. It will require several preparatory steps.

Given $\chi\in\cC^\infty_c((0,\infty))$ and $h>0$ define:
\[
\cF^h_\chi\,:\,L^2(M)\To C(\IR;L^2(M))\,:\,u^0\longmapsto \chi(-h^2\Delta) u,
\]
where $u$ solves \eqref{e:fracschro} with initial datum $u^0$.
\begin{lem}\label{l:ape}
Let $\chi\in\cC^\infty_c((0,\infty))$ and $\sigma_\alpha\in\cC^\infty_c((0,\infty))$ such that $\sigma_\alpha(s)=s^{\alpha/2}$ for $s\in\supp\chi$. Then for every $u^0\in L^2(M)$ and $h>0$  the function $u_h:=\cF^h_\chi u^0$ solves
\begin{equation}\label{e:sfracschro}
\left\{
\begin{array}{l}
ih^\alpha\partial_t u_h(t,x)=\sigma_\alpha(-h^2\Delta )u_h(t,x)+h^\alpha V(t,x)u_h(t,x)+h^{1+\alpha}F_h(t,x), \medskip\\
u_h|_{t=0}=\chi(-h^2\Delta)u^0,
\end{array}
\right.
\end{equation}
and there exists $C>0$ such that for every $t\in\R$ and $h\in (0,1]$,
\begin{equation}\label{e:estFh}
\|F_h(t,\cdot)\|_{L^2(M)}\leq C\|u^0\|_{L^2(M)}.
\end{equation}
If $u$ is the solution to \eqref{e:fracschro} with initial datum $u^0$ then
\begin{equation}\label{e:estuuh}
 \|u(t,\cdot)-u_h(t,\cdot)\|_{L^2(M)}\leq \|(1-\chi(-h^2\Delta) )u^0\|_{L^2(M)}+Ch|t|\|u^0\|_{L^2(M)}.
\end{equation}
\end{lem}
\begin{proof}
Apply $\chi(-h^2\Delta)$ to both sides of equation \eqref{e:fracschro} to find that $u_h$ solves:
\[
i\partial_t u_h=(-\Delta )^{\alpha/2}u_h+ Vu_h+[\chi(-h^2\Delta),V]u.
\]
By definition 
\[
(-h^2\Delta )^{\alpha/2}\chi(-h^2\Delta)=\sigma_\alpha(-h^2\Delta )\chi(-h^2\Delta),
\]
and, using the commutator identity \eqref{e:comm} and the functional calculus for the Laplacian \eqref{e:fc},
\[
[\chi(-h^2\Delta),V]u=\frac{h}{i}\Op_h(r)u,
\]
for some $r\in S^{0}(T^*M)$. This proves claims \eqref{e:sfracschro} and \eqref{e:estFh}. Estimate \eqref{e:estuuh} follows from Duhamel's identity.
\end{proof}

The rest of this section will show how the dynamics of $u_h$ when $h$ is small are related to the geodesic flow on the manifold $(M,g)$. We introduce the following notation for the squared Riemannian norm on the cotangent bundle
\[
p(x,\xi):=\|\xi\|^2_x,\quad (x,\xi)\in T^*M.
\]
The geodesic flow $\phi_t$ on $T^*M\setminus\{0\}$ is the flow of the Hamiltonian vector field (taken with respect to the canonical symplectic form) associated to the Hamiltonian $p^{1/2}$. Projections of the trajectories of $\phi_t$ onto $M$ are geodesics of $(M,g)$ parametrized by arc-length.

We next deduce a transport-type equation for the Wigner distributions of solutions to equation \eqref{e:sfracschro} involving the Hamiltonians $p^{\alpha/2}$. Recall that the Wigner distribution $W^h_v\in\cD'(T^*M)$ of a function $v\in L^2(M)$ is defined by
\[
\left\la W^h_v,a\right\ra = (\Op_h(a) v\,|\,v)_{L^2(M)},\quad \forall a\in\cC^\infty_c(T^*M),
\]
where $\Op_h(a)$ stands for the Weyl semiclassical pseudo-differential operator of symbol $a$. Appendix \ref{a:pdo} reviews the facts of the theory of Wigner distributions and semiclassical analysis that are needed here. In what follows, $\{a,b\}$ will denote the Poisson bracket of two functions $a,b$ defined on $T^*M$.
\begin{lem}\label{l:wie}
Let $W_h(t)\in\cD'(T^*M)$ denote the Wigner distribution of the function $\cF^h_\chi u^0(t,\cdot)$ for some $u^0\in L^2(M)$. Then, for every $a\in\cC^\infty_c(T^*M)$:
\begin{equation}\label{e:wignereq}
\frac{d}{dt}\left\la W_h(t),a\right\ra = h^{1-\alpha}\left\la W_h(t),\{\sigma_\alpha\circ p,a\}\right\ra+h^\beta R^h_a(t),
\end{equation}
where $\beta:=\min\{1,2-\alpha\}$ and $|R^h_a(t)|\leq C\|a\|_{\cC^N(T^*M)} \|u^0\|^2_{L^2(M)}$ for some $N>0$ large enough, independent of $a$.
\end{lem}
\begin{proof}
Set $u_h:=\cF^h_\chi u^0$. By definition of the Wigner distribution and equation \eqref{e:sfracschro}:
\begin{equation}\label{e:wignertem}
\frac{d}{dt}\left\la W_h(t),a\right\ra =\frac{1}{ih^\alpha}([\Op_h(a),\sigma_\alpha(-h^2\Delta )] u_h(t,\cdot)|u_h(t,\cdot))_{L^2(M)}+D^h_a(t),
\end{equation}
with 
\[
D^h_a(t)=i([V,\Op_h(a)]u_h(t,\cdot)|u_h(t,\cdot))_{L^2(M)}+2h\mathrm{Im}(\Op_h(a)F_h(t,\cdot)|u_h(t,\cdot))_{L^2(M)}.
\]
Now using the symbolic calculus of semiclassical pseudo-differential operators \eqref{e:comm} and the functional calculus \eqref{e:fc}, we deduce that:
\[
[\Op_h(a),\sigma_\alpha(-h^2\Delta )] = \frac{h}{i}\Op_h(\{a,\sigma_\alpha\circ p\})+h^2\Op(r),
\]
for some symbol $r\in S^{0}(T^*M)$ (keep in mind that symbol $\sigma_\alpha\circ p$ is a function in $\cC^\infty_c(T^*M)$). Substituting this in \eqref{e:wignertem} gives \eqref{e:wignereq} with 
\[
R^h_a(t)=h^{-\beta}D^h_a(t)+ih^{2-\alpha-\beta}(\Op_h(r)u_h(t,\cdot)|u_h(t,\cdot))_{L^2(M)}.
\]
The estimate then follows from the Calderón-Vaillancourt theorem \eqref{e:cv} and again from \eqref{e:comm}, that allows us to estimate $|D^h_a(t)|\leq Ch \|a\|_{\cC^N(T^*M)} \|u^0\|^2_{L^2(M)}$ for some $N\geq N_d$ ($N_d$ appears in \eqref{e:cv}), which is independent of $a$.
\end{proof}

This result will allow us to characterize semiclassical measures of solutions to \eqref{e:fracschro}. Recall that a semiclassical measure of a sequence $(v_h)_{0<h\leq 1}$ that is bounded in $L^2(M)$ is an accumulation point in $\cD'(T^*M)$ of the corresponding sequence of Wigner distributions $(W^h_{v_h})_{0<h\leq 1}$.

\begin{prop}\label{p:scm}
Let $(u^0_h)_{0<h\leq 1}$ be bounded in $L^2(M)$ and $u_h:=\cF^h_\chi u^0_h$. Denote by  $(W_h(t))_{0<h\leq 1}$  the sequence of Wigner distributions of $(u_h(t,\cdot))_{0<h\leq 1}$. The following hold:
\begin{itemize}
\item if $\alpha\in (0,1]$ then there exists a subsequence $(h_n)$ converging to zero along which $(W_{h_n}(t))$ converges in $\cD'(T^*M)$ for every $t\in\IR$ to a semiclassical measure $\mu_t$. In addition:
\[
\mu_t=\left\{ \begin{array}{ll}
\mu_0 & \text{ if }0<\alpha<1,\medskip\\
(\phi_t)_*\mu_0 & \text{ if } \alpha =1;
\end{array} \right.
\]
\item if $\alpha>1$ and $(W_{h_n})$ converges in $\cD'(\IR\times T^*M)$ to a semiclassical measure $\mu_t$ as in \eqref{e:tscm} then, for almost every $t\in\IR$:
\[
\mu_t=(\phi_s)_*\mu_t,\quad \forall s\in\IR.
\]
\end{itemize}
\end{prop}
\begin{proof}
When $\alpha\in (0,1]$, equation \eqref{e:wignereq} shows that $(W_{h})_{0<h\leq 1}$ is uniformly bounded in $\cC^1(\IR;\cD'(T^*M))$. Therefore, there exists a subsequence $(h_n)$ converging to zero such that $(W_{h_n}(t))$ converges for every $t\in\IR$. The resulting accumulation points $\mu_t$ satisfy, as a consequence of  \eqref{e:wignereq}, that for every $a\in\cC^\infty_c(T^*M)$
\[
\frac{d}{dt}\int_{T^*M}a(x,\xi)\mu_t(dx,d\xi)=0,\quad \text{ for }0<\alpha<1,
\]
which concludes the proof in this case, and
\[
\frac{d}{dt}\int_{T^*M}a(x,\xi)\mu_t(dx,d\xi)=
\int_{T^*M}\{a,\sigma_1\circ p\}(x,\xi)\mu_t(dx,d\xi),
\quad\text{ for }\alpha=1.
\]
Notice that, by construction of the sequence $(u_h)_{0<h\leq 1}$, on the support of $\mu_t$ one has $\sigma_1\circ p=p^{1/2}$. Therefore $\mu_t$ is transported along the Hamiltonian flow $p^{1/2}$ which is the geodesic flow $\phi_t$.

Suppose now that $\alpha>1$. In this case $(W_{h})_{0<h\leq 1}$ is only bounded in  $\cC(\IR;\cD'(T^*M))$, and therefore one cannot expect to have pointwise convergence in $t\in\IR$.  Still $(W_{h})_{0<h\leq 1}$ has accumulation points in $\cD'(\IR\times T^*M)$ (see Remark \ref{r:tbdd}), that are described by \eqref{e:tscm}. In our case this yields: 
\[
\int_\IR\theta(t)\left\la W_{h_n}(t),a\right\ra dt \Tend{n}{\infty} \int_\IR\int_{T^*M}\theta(t)a(x,\xi)\mu_t(dx,d\xi)dt,
\]
for all $\theta\in \cC^\infty_c(\IR)$, $a\in\cC^\infty_c(T^*M)$. 

Note that, after multiplication by $\theta\in \cC^\infty_c(\IR)$ and integration by parts, \eqref{e:wignereq} can be rewritten as:
\[
-h^{\alpha-1}_n\int_\IR\theta'(t)\left\la W_{h_n}(t),a\right\ra dt=\int_\IR\theta(t)\left\la W_{h_n}(t),\{\sigma_\alpha\circ p,a\}\right\ra dt+\cO(h_n).
\]
Taking limits and noticing that, as before,  $\sigma_\alpha\circ p=p^{\alpha/2}$ on the support of $\mu_t$ for almost every $t\in\IR$, we deduce that:
\[
\int_\IR\int_{T^*M}\theta(t)\{a, p^{\alpha/2}\}(x,\xi)\mu_t(dx,d\xi)dt=0,\quad \forall \theta\in \cC^\infty_c(\IR),\,\forall a\in\cC^\infty_c(T^*M).
\]
This identity implies that $\mu_t$ is, for a.e. $t\in\IR$, invariant by the Hamiltonian flow associated to $p^{\alpha/2}$, which is merely a reparametrization of the geodesic flow $\phi_t$. The claim then follows.
\end{proof}

Let $(x_0,\xi_0)\in S^*M$ (recall that this means that $\|\xi_0\|_{x_0}=1$). A wave-packet centered at $(x_0,\xi_0)$ is a family $(u_h^{(x_0,\xi_0)})_{0<h\leq 1}$ of functions supported in a coordinate chart $(U,\varphi)$ of $M$ with $x_0\in U$ of the form: 
\[
u_h^{(x_0,\xi_0)}(x)=\frac{1}{h^{d/4}}\rho\left(\frac{\varphi(x)-\varphi(x_0)}{\sqrt{h}}\right)e^{i\frac{(d\varphi_{x_0}^{-1})^T\xi_0}{h}\cdot \varphi(x)},\quad x\in U,
\]
where $\rho\in \cC^\infty_c(\varphi(U))$ and normalized to have $\|u^{(x_0,\xi_0)}_h\|_{L^2(M)}=1$.

Let $\chi\in\cC^\infty_c((0,\infty);[0,1])$ such that $\chi|_{[1,2]}=1$, and $\chi(s)=0$ for $s<1/2$ or $s>5/2$ and write $\Pi_h:=\chi(-h^2\Delta)$.
\begin{lem}\label{l:wp}
For every $(x_0,\xi_0)\in S^*M$, the sequence $(W^{(x_0,\xi_0)}_h)_{0<h\leq 1}$ of Wigner distributions of a frequency-localized wave-packet $(\Pi_h u_h^{(x_0,\xi_0)})_{0<h\leq 1}$ converges in $\cD'(T^*M)$ to the Dirac mass
$\delta_{(x_0,\xi_0)}$.
\end{lem}
\begin{proof}
A direct computation shows that the Wigner distributions 
 of $u_h^{(x_0,\xi_0)}$ converge to $\delta_{(x_0,\xi_0)}$. Identities \eqref{e:prod} and \eqref{e:fc} imply that $(W^{(x_0,\xi_0)}_h)_{0<h\leq 1}$ converges to
\[
(\chi\circ p)^2\delta_{(x_0,\xi_0)}=\delta_{(x_0,\xi_0)},
\]
as claimed.
\end{proof}
We now give the proofs of Theorems \ref{t:tdep} and \ref{t:suff}.
\begin{proof}[Proof of Theorem \ref{t:tdep}]
We start with the case $0<\alpha<1$. Let $x_0$ be a point in the open set $M\setminus\overline{\omega}$, and 
\[
u^0_h:=\Pi_h u_h^{(x_0,\xi_0)},\quad \xi_0\in S^*_{x_0}M.
\] 
Let $u_h$ be the solution to \eqref{e:fracschro} issued from $u^0_h$. If $\chi$ is the cut-off used to define $\Pi_h$ then $\Pi_h u_h=\cF^h_\chi u^0_h$ and Lemma \ref{l:ape} implies that:
\begin{equation}\label{e:aeq}
\lim_{h\to 0^+}\|u_h(t,\cdot)-\Pi_h u_h(t,\cdot)\|_{L^2(M)}=0,\quad \forall t\in\IR.
\end{equation}
Let $(h_n)$ be a sequence given by applying Proposition \ref{p:scm} to $(u^0_h)_{0<h\leq 1}$ with the cut-off $\chi$ and let $\mu_t$ the corresponding semiclassical measure. By Lemma \ref{l:wp} and Proposition \ref{p:scm} we know that $\mu_t=\delta_{(x_0,\xi_0)}$ for every $t\in\IR$. Combining this with \eqref{e:proj} and \eqref{e:aeq} we find that, for every $T>0$ and every $b\in\cC(M;[0,1])$ that equals $1$ on $\overline{\omega}$ and vanishes in a neighborhood of $x_0$,
\begin{align*}
\lim_{n\to\infty}\int_0^T\int_\omega|u_{h_n}(t,x)|^2dx\,dt\leq & \lim_{n\to\infty}\int_0^T\int_M b(x)|\Pi_{h_n} u_{h_n}(t,x)|^2dx\,dt\\=&
\int_0^T \int_{T^*M} b(x)\delta_{(x_0,\xi_0)}(dx,d\xi)dt=0.
\end{align*}
Since $\|u^0_{h_n}\|_{L^2(M)}=1$, this shows that no constant $C>0$ can exist such that \eqref{e:obst} holds.

Let us now consider the case $\alpha=1$. The proof follows closely the lines of that of the preceding case with few modifications. Choose now $(x_0,\xi_0)\in S^*M$ such that $\phi_t(x_0,\xi_0)\not\in\omega$ for $t\in [0,T]$. Proposition \ref{p:scm} now shows that the semiclassical measures of the corresponding wave-packets are $\mu_t=\delta_{\phi_t(x_0,\xi_0)}$. The same argument we presented above shows that these wave-packets violate any observability inequality of the form \eqref{e:obst}.
\end{proof}
\begin{proof}[End of the proof of Theorem \ref{t:suff}]
It remains to prove that GCC$_T$ (resp. GCC) are sufficient for \eqref{e:obss} when $\alpha=1$ (resp. $\alpha>1$). Since $V$ does not depend on time, we can use frequency localization and unique continuation for eigenfunctions of Schrödinger operators as in \cite[Proof of Theorem4.1]{Lebeau1992} (see also \cite[Proposition 4.1]{BZ12}) to show that \eqref{e:obss} can be deduced from the \textit{a priori} weaker semiclassical estimate: there exist $h_0>0$ such that, for every $u^0\in L^2(M)$ and $0<h<h_0$,
\begin{equation}\label{e:sobs}
\|\Pi_h u^0\|_{L^2(M)}^2\leq C\int_0^T\int_\omega | e^{-it((-\Delta)^{\alpha/2}+V)} \Pi_h u^0|^2dx\,dt.
\end{equation}
We prove that \eqref{e:sobs} holds by contradiction. If  \eqref{e:sobs} fails, then it is possible to find a sequence $(h_n)$ that tends to zero and functions $u^0_{h_n}\in L^2(M)$ such that:
\begin{equation}\label{e:contr}
\|\Pi_{h_n} u^0_{h_n}\|_{L^2(M)}=1,\quad \lim_{n\to\infty}\int_0^T\|e^{-it((-\Delta)^{\alpha/2}+V)} \Pi_{h_n} u^0_{h_n}\|_{L^2(\omega)}=0.
\end{equation}
Modulo the extraction of a subsequence, we can assume that $(W^{h_n}_{u_{h_n}})$ converges to a semiclassical measure $\mu_t$, which by \eqref{e:proj}, \eqref{e:prob} and \eqref{e:contr} satisfies:
\[
\int_0^T\mu_t(T^*M\setminus\{0\})dt=T,\quad\int_0^T\int_{T^*M}b(x)\mu_t(dx,d\xi)dt=0,\quad \forall b\in\cC_c(\omega).
\]
Using Lemma \ref{l:ape} and Proposition \ref{p:scm} we deduce, writing:
\[
F^T_\omega:=\bigcup_{t\in [0,T]}\{\phi_t(x,\xi)\,:\, (x,\xi)\in T^*\omega\setminus\{0\}\},
\]
that $\mu_0(F^T_\omega)=0$ when $\alpha=1$ and $\mu_t(F^s_\omega)=0$ for every $s>0$ and almost every $t\in [0,T]$ when $\alpha>1$. If $\alpha=1$ and $\omega$ satisfies GCC$_T$ this implies $\mu_0(T^*M\setminus\{0\})=0$; whereas if $\alpha>1$ and $\omega$ satisfies GCC it gives $\mu_t(T^*M\setminus\{0\})=0$ for almost every $t\in [0,T]$. This is a contradiction, and the result is proved.
\end{proof}

\section{Observability on the sphere. Proof of Theorem \ref{t:nes}}\label{s:nec}

Here we focus on the particular case $(M,g)=(\IS^d,\mathrm{can})$ and prove, at the end of this section, Theorem \ref{t:nes}. The geodesics on the sphere are great circles, obtained as the intersection of $\IS^d$ with planes through the origin. We normalize the Riemannian metric in order that all the geodesics of the sphere have length equal  to $2\pi$. Therefore, the geodesic flow $\phi_t$ is periodic of period $2\pi$.
 
The \textit{Radon transform} or \textit{X-Ray transform} of a function $a\in\cC^\infty(T^*\IS^d)$ is defined by averaging $a$ along orbits of the geodesic flow:
\begin{equation}\label{e:rt}
\cI(a)(x,\xi)=\frac{1}{2\pi}\int_0^{2\pi}a(\phi_s(x,\xi)) ds, \quad (x,\xi)\in T^*\IS^d\setminus\{0\}.
\end{equation}
As soon as $a\in\cC^\infty_c(T^*\IS^d\setminus\{0\})$ (in particular, $a$ vanishes in a neighborhood of the zero section $\{\xi=0\}$) one can identify $\cI(a)$ to a function in $\cC^\infty_c(T^*\IS^d\setminus\{0\})$ as well. 

Recall that the spectrum of $-\Delta$ is:
\[
\Spec(-\Delta)=\{k(k+d-1)\,:\,k\in\N\cup\{0\}\}.
\]
Let $A$ be a positive, self-adjoint operator such that
\[
A^2:= -\Delta+\frac{(d-1)^2}{4}; 
\]
then $A=\Op_\hbar(\sqrt{p+(d-1)^2/4})+O(h)$  and the spectrum of $A$ equals $\IN+(d-1)/2$.   
Then the unitary flow generated by $A$ is essentially periodic of period $2\pi$:
\begin{equation}\label{e:ap}
e^{2i\pi A}=e^{i\pi(d-1)}\text{Id}.
\end{equation}
Given $a$ in $\ml{C}^{\infty}_c(T^*\IS^d\setminus\{0\})$, we define following Weinstein \cite{WeinsteinZoll} the quantum average of the operator $\Op_h(a)$:
$$\cI_{\text{qu}}(\Op_h(a)):=\frac{1}{2\pi}\int_0^{2\pi}e^{-is A}\Op_h(a)e^{isA}ds.$$
Then, it follows from \eqref{e:ap} that
\begin{equation}\label{e:commav}
\left[\ml{I}_{\text{qu}}(\Op_h(a)),A\right]=\left[\ml{I}_{\text{qu}}(\Op_h(a)),(-\Delta)^{\alpha/2}\right]=0.
\end{equation}
We define the averaged Wigner distribution $\cI^*(W^h_u)$ of a function $u\in L^2(\IS^d)$ as:
\[
\left\la \cI^*(W^h_u), a \right\ra :=(\cI_{\text{qu}}(\Op_h(a)) u\,|\, u)_{L^2(\IS^d)} , \quad \forall a\in\ml{C}^{\infty}_c(T^*\IS^d\setminus\{0\}).
\] 
One has 
\begin{equation}\label{e:wws}
| \left\la \cI^*(W^h_u), a \right\ra - \left\la W^h_u, \cI(a) \right\ra |\leq Ch\|\cI(a)\|_{\cC^N(T^*\IS^d)}\|u\|_{L^2(\IS^d)}^2,
\end{equation}
for some $N\geq N_d$, $N_d$ being defined in the Calderón-Vaillancourt Theorem \eqref{e:cv}, that is  independent of $a$. This is a consequence of Egorov's theorem \cite{DimassiSjostrand,F14,Zwobook}, which implies:
\begin{equation}\label{e:egorov}
 \ml{I}_{\text{qu}}(\Op_h(a))=\Op_h(\ml{I}(a))+h\Op_h(r),
\end{equation} 
for some $r\in S^0(T^*\IS^d)$.
\begin{lem}\label{l:ws}
Let $u^0\in L^2(\IS^d)$ and $u$ be the solution to \eqref{e:fracschro} with initial datum $u^0$. Then, for every $t\in\IR$:
\begin{equation}
\left\la \cI^*(W^h_{u(t,\cdot)}), a \right\ra =  \left\la \cI^*(W^h_{u^0}), a \right\ra + hR^h_a(t) ,
\end{equation} 
where $|R^h_a(t)|\leq C|t|\|\cI(a)\|_{\cC^N(T^*\IS^d)}\|u^0\|_{L^2(\IS^d)}^2$ for some $N>0$ which is independent of $a$.
\begin{proof}
Using \eqref{e:commav} we deduce that, for every $t\in\IR$ and $\ml{C}^{\infty}_c(T^*\IS^d\setminus\{0\})$,
\begin{align*}
\frac{d}{dt}\left\la \cI^*(W^h_{u(t,\cdot)}),a\right\ra =&\frac{1}{i}([\cI_{\text{qu}}(\Op_h(a)),(-\Delta)^{\alpha/2}+V(t,\cdot)] u(t,\cdot)|u(t,\cdot))_{L^2(M)}\\
=&i([V(t,\cdot),\cI_{\text{qu}}(\Op_h(a))] u(t,\cdot)|u(t,\cdot))_{L^2(M)}.
\end{align*}
To conclude, use \eqref{e:egorov} and \eqref{e:comm} and deduce:
\[
[V(t,\cdot),\cI_{\text{qu}}(\Op_h(a))]=h\Op_h(r_t),
\]
for some $r_t\in S^0(T^*\IS^d)$ uniformly bounded with respect to $t\in\IR$. 
\end{proof}
\end{lem}

\begin{proof}[Proof of Theorem \ref{t:nes}] 
In order to prove the first statement about the fractional Schrödinger evolution, we take $(x_0,\xi_0)\in S^*M$ such that $\phi_t(x_0,\xi_0)\not\in S^*\omega$ for every $t\in\IR$. Consider the sequence $(u_h)$ of solutions to \eqref{e:fracschro} issued from the initial data $(\Pi_h u^{(x_0,\xi_0)}_h)$. We know by Lemma \ref{l:wp} that  $(u_h(0,\cdot))$ has as semiclassical measure $\delta_{(x_0,\xi_0)}$. We use \eqref{e:wws} and Lemma \ref{l:ws} to deduce that the semiclassical measure $\mu_t$ of the sequence $(u_h)$ satisfies for every $\theta\in \cC^\infty_c(\IR)$ and $a\in\cC^\infty_c(T^*M\setminus\{0\})$:
\[
\int_\IR\int_{T^*M}\theta(t)\cI(a)(x,\xi)\mu_t(dx,d\xi)dt=\left(\int_\IR\theta(t) dt\right)\int_{T^*M}\cI(a)(x,\xi)\mu_0(dx,d\xi).
\]
In other words, $\mu_t|_{T^*M\setminus\{0\}}=\cI^*(\mu_0)|_{T^*M\setminus\{0\}}$ for almost every $t\in\IR$. 

On the other hand, we know by Proposition \ref{p:scm} that, for almost every $t\in\IR$,  $\mu_t$ is invariant by the geodesic flow. Therefore, 
\begin{align*}
\int_\IR\int_{T^*M}\theta(t)a(x,\xi)\mu_t(dx,d\xi)dt=&\int_\IR\int_{T^*M}\theta(t)\cI(a)(x,\xi)\mu_t(dx,d\xi)dt\\
=& \left(\int_\IR\theta(t) dt\right)\cI(a)(x_0,\xi_0).
\end{align*}
If $b\in\cC(M;[0,1])$ is such that $b$ vanishes in a neighborhood of the geodesic issued from $(x_0,\xi_0)$ and is equal to one on $\omega$ we conclude, as in the proof of Theorem \ref{t:tdep}, that along some subsequence:
\[
\lim_{n\to\infty}\int_0^T\int_\omega|u_{h_n}(t,x)|^2dx\,dt\leq T\cI(b)(x_0,\xi_0)=0.
\]
Since $\|u_{h_n}(0,\cdot)\|_{L^2(M)}=1$ we conclude that no constant $C>0$ exists such that \eqref{e:obss} holds.\medskip

We now prove the statement concerning eigenfunctions of $(-\Delta)^{\alpha/2}$. Since these eigenfunctions do not depend on $\alpha$, we will assume that $\alpha=2$. 

Write the sphere as:
$$\mathbb{S}^d:=\{x\in\IR^{d+1}:|x|=1\}.$$ 
Let
$$\varphi_k(x)=c_k(x_1+ix_2)^k, \quad \text{ with } c_k:=\sqrt{\frac{\Gamma(k+(d+1)/2)}{2\pi^{\frac{d+1}{2}}k!}}\sim k^{\frac{d-1}{4}}.$$
This function is a spherical harmonic and therefore an eigenfunction of the Laplacian:
$$-\Delta \varphi_k(x)=k(k+d-1)\varphi_k(x),\quad x\in\mathbb{S}^d,\quad ||\varphi_k||_{L^2(\mathbb{S}^d)}=1.$$ 
Clearly
$$|\varphi_k(x)|^2=(c_k)^2(|x_1|^2+|x_2|^2)^k=(c_k)^2(1-|x'|^2)^k,$$
where $x=(x_1,x_2,x')$. This shows that $|\varphi_k|^2$ concentrates on the equator $\{x'=0\}$.

If $\overline{\omega}\cap\{x'=0\}=\emptyset$ then no constant $C>0$ can exist such that
\begin{equation*}
||\varphi_k||_{L^2(\IS^d)}\leq C \|\varphi_k\|_{L^2(\omega)},
\end{equation*}
holds uniformly in $k\in\IN$, since
$$\lim_{k\to\infty}\int_\omega |\varphi_k(x)|^2dx=0,\quad \text{ and }\quad ||\varphi_k||_{L^2(\mathbb{S}^d)}=1.$$
Since any other geodesic of $\mathbb{S}^d$ can be obtained by applying a rotation to $\{x'=0\}$, and the composition of a spherical harmonic with an Euclidean rotation is again a spherical harmonic,  the claim follows.
\end{proof}

\section{Observability over long times and the $V$-GCC}\label{s:vgcc}

As we mentioned in the introduction, observability for eigenfunctions of the fractional Schrödinger operator $(-\Delta)^{\alpha/2}+V$ holds under a geometric assumption on $\omega$ that involves the perturbation $V$. Let us start by describing this new condition. 

Identify $V$ to a smooth function on $T^*\IS^d$ that does not depend on the cotangent variable and consider its Radon transform $\cI(V)\in\cC^\infty_c(T^*\IS^d\setminus\{0\})$.
The function $\cI(V)$ defines a Hamiltonian vector field $X_{\cI(V)}$ on $T^*\IS^d\setminus\{0\}$ (with respect to the canonical symplectic form in $T^*\IS^d$). Its flow $\phi^V_s$ commutes with the geodesic flow, since by construction $\cI(V)$ is invariant by the geodesic flow, and therefore $\{\cI(V),p^{1/2}\}=0$. 

As a consequence, $\phi^V_s$ maps orbits of the geodesic flow into orbits of the geodesic flow. In other words, for every geodesic $\gamma_0\subseteq T^*\IS^d$
\[
\gamma_s:=\phi^V_s(\gamma_0),
\]
is also a geodesic for every $s\in\IR$. We can thus identify $\phi^V_s$ to a function acting on the space of geodesics on $T^*\IS^d\setminus\{0\}$.

This flow on the space of geodesics induces a new geometric condition on $\omega$, that we name the $V$-Geometric Control Condition, that holds provided that
\begin{equation}\label{e:vgcc}\tag{$V$-GCC$_T$}
 K^{V}_{T,\omega}:=\{\gamma_0\text{ geodesic} \;:\; \phi_s^V(\gamma_0)\cap T^*\omega\neq \emptyset,\;\forall s\in (0,T)\}=\emptyset.
\end{equation}
In other words $\omega$ satisfies \eqref{e:vgcc} provided that, given any geodesic  $\gamma_0$ one can find $s\in (0,T)$ such that $\gamma_s\cap \omega\neq \emptyset$. 

This new condition is sufficient in order to have observability for the fractional Schrödinger evolution over long times.
\begin{theo}\label{t:lt}
Suppose $(M,g)=(\IS^d,\mathrm{can})$ and that $\omega\subseteq \IS^d$ is open and $\alpha>0$. A constant $C>0$ exists such that the time-frequency observability estimate
\begin{equation}\label{e:lsobs}
\|\Pi_h u^0\|_{L^2(M)}^2\leq \frac{C}{T_h}\int_0^{T_h}\int_\omega | e^{-it((-\Delta)^{\alpha/2}+V)} \Pi_h u^0|^2dx\,dt,
\end{equation} 
holds for every $u^0_h\in L^2(M)$ and $h\in (0,h_0]$ in these two cases:
\begin{itemize}
\item $T_h=T/h$ and $\omega$ satisfies $V$-$GCC_{T}$.
\item $hT_h\to \infty$ as $h\to 0 ^+$ and $\omega$ satisfies $V$-$GCC_{T}$ for some $T>0$.
\end{itemize}
\end{theo}
Since solutions issued from an eigenfunction are periodic in time, we obtain the following consequence, which is proved in \cite{MacRiv19} when $\alpha=2$.
\begin{coro}\label{c:eig}
Suppose that $\omega$ satisfies $V$-$GCC_{T}$ for some $T>0$. Then \eqref{e:obse} holds.
\end{coro}
\begin{rem}  
If $V$ is odd (meaning $V(x)=-V(-x)$) then  $\cI(V)=0$ and $V$-GCC$_{T}$ is equivalent to GCC. However, in \cite{MacRiv19} it is shown that a similar result to Corollary \ref{c:eig} holds under a new geometric condition, in which the Radon transform of the potential is replaced by a different nonlinear transform of $V$, whose expression is a bit more complicated. This term could be constant again, and in general one gets a geometric condition related to the Hamiltonian flow of the first non-vanishing term in a Quantum Birkhoff Normal Form (see \cite{ArM20} for a precise account on the closely related case of the Harmonic Oscillator). Up to our knowledge, it is not known whether or not the vanishing of all those terms implies that $V$ is constant.

This yields the following question: suppose that \eqref{e:obse} holds for every eigenfunction of $-\Delta+V$  if and only if $\omega$ satisfies GCC. Does this imply that $V$ is constant?
\end{rem}
\begin{proof}[Proof of Theorem \ref{t:lt}]
We argue again by contradiction: suppose \eqref{e:lsobs} is not true, this means that there exist a sequence $(h_n)$ that tends to zero and functions $u^0_{h_n}\in L^2(M)$ such that $\|\Pi_{h_n} u^0_{h_n}\|_{L^2(M)}=1$ and
\begin{multline}
\lim_{n\to\infty}\frac{1}{T_{h_n}}\int_0^{T_{h_n}}\|e^{-it((-\Delta)^{\alpha/2}+V)}\Pi_{h_n} u^0_{h_n}\|^2_{L^2(\omega)}dt\\=\lim_{n\to\infty}\int_0^1\|e^{-itT_{h_n}((-\Delta)^{\alpha/2}+V)} \Pi_{h_n} u^0_{h_n}\|^2_{L^2(\omega)}dt=0.
\end{multline}
Consider the semiclassical measure $\mu_t$ of the sequence $(u_{h_n})$ where
\[
u_{h_n}(t,\cdot):=e^{-itT_{h_n}((-\Delta)^{\alpha/2}+V)}\Pi_{h_n} u^0_{h_n}. 
\]
A straightforward modification of the proof of Proposition \ref{p:scm} gives that, for almost every $t\in\IR$, the measures $\mu_t$ are invariant by the geodesic flow. Moreover, it is proved in \cite[Proposition 2.2]{MacRiv16} that the following hold when $\alpha=2$:
\begin{itemize}
\item if $T_h=T/h$ then $\mu_t=(\phi^V_{tT})\cI^*(\mu_0)$, where $\mu_0$ stands for the semiclassical measure of $(u_{h_n}(0,\cdot))$;
\item if $hT_h\to\infty$ then $(\phi^V_t)_*\mu_t=\mu_t$ for almost every $t\in\IR$.
\end{itemize}
To see why this holds for every $\alpha>0$, simply recall the proof of Lemma \ref{l:ws}, taking now into account that time derivatives of the Wigner distributions give a new factor $T_{h_n}$: 
\[
\frac{d}{dt}\left\la \cI^*(W^{h_n}_{u_{h_n}(t,\cdot)}),a\right\ra=T_{h_n}h_n\left\la W^{h_n}_{u_{h_n}(t,\cdot)},\{V,\cI(a)\}\right\ra+\cO(T_{h_n}h_n^2).
\]
By taking limits, it follows that, for every $t\in\R$,
\[
\left\la \mu_t,\cI(a)\right\ra-\left\la \mu_0,\cI(a)\right\ra=T\int_0^t \left\la \mu_s,\{V,\cI(a)\}\right\ra ds, \quad\text{ when }T_{h_n}h_n=T,
\]
and, for every $\theta\in\cC^\infty_c(\IR)$,
\[
0=\int_\IR\theta(t)\left\la \mu_t,\{V,\cI(a)\}\right\ra dt, \quad\text{ when }T_{h_n}h_n\to\infty.
\]
Since the measure $\mu_t$ is invariant by the geodesic flow,
\[
\left\la \mu_t,a\right\ra=\left\la \mu_t,\cI(a)\right\ra,
\]
and
\[
\left\la \mu_t,\{V,\cI(a)\}\right\ra=\left\la \mu_t,\{\cI(V),\cI(a)\}\right\ra.
\]
From here the claimed invariance and transport properties of $\mu_t$ follow (for further details see the proofs of \cite[Proposition 2.2]{MacRiv16} and Proposition \ref{p:scm} in Section \ref{s:suf}).
Now, since $\mu_t(T^*\omega\setminus\{0\})=0$, $t\in [0,1]$ and $\omega$ satisfies $V$-GCC$_{T}$ it follows that $\mu_t=0$ for almost every $t\in\IR$, which is a contradiction.
\end{proof}

\section{Proof of Theorem \ref{t:minimal}}\label{s:potential}
The Radon transform of the potential  $\cI(V)$ is always a zero-homogeneous smooth function on $T^*M\setminus \{0\}$; as such, it can be identified to a unique element in $\cC^\infty(S^*\IS^d)$. In addition, since  $\cI(V)$ is invariant by the geodesic flow, it can be identified to a function on $G(\IS^2)$, the space of oriented geodesics. Recall that $G(\IS^d)$ can be constructed as the quotient space of $S^*\IS^d$ in which $(x,\xi)$ and $(x',\xi')$ are equivalent if and only if they belong to the same orbit of the geodesic flow.

The two-dimensional sphere has the nice feature that its space of oriented geodesics $G(\IS^2)$ can be identified to the sphere $\IS^2$ itself and the symplectic form on $T^*\IS^2$ induces a symplectic structure on $G(\IS^2)$ (which must necessarily be a non-zero multiple of the volume form on $\IS^2$).  The mapping:
\[
\Phi: G(\IS^2)\To \IS^2 : \gamma \longmapsto x\times \xi,
\]
where $(x,\xi)\in \gamma$ and $\times$ denotes the vector product in $\IR^3$, is well defined and bijective. To see this, note that every geodesic in $\IS^2$ is obtained by intersecting the sphere by the linear plane spanned by $x$ and $\xi$, where $x$ is a point on the geodesic and $\xi$ a unitary cotangent vector to the geodesic at $x$. The two unit normal vectors of this plane are obtained as $x\times \xi$, depending on the choice of orientation of $\xi$. For instance, if $\gamma^\pm$ denotes the geodesic $\{x_3=0\}$ parametrized positively/negatively from the point of view of an observer located at $(0,0,1)$ then:
\[
\Phi(\gamma^\pm)=(0,0,\pm 1).
\] 
The set $G_{x_0}$ of all geodesics issued from the same point $x_0\in\IS^2$ is then mapped via $\Phi$ to the geodesic in $\IS^2$ that lies in the plane through the origin that is orthogonal to $x_0$. $G(\IS^2)$ has natural smooth and symplectic structures inherited from $T^*M$, which are preserved by $\Phi$.

With this in mind, the Radon transform, when restricted to functions of $\cC^\infty(S^*\IS^2)$ that only depend on $x$, can be identified to an operator:
\[
\tilde{\cI}\,:\, C^\infty(\IS^2)\To  C^\infty(\IS^2),
\]
where $\tilde{\cI}=(\Phi^*)^{-1}\circ \cI$. Then, see \cite{Guillemin1976},
\[
\ker \tilde{\cI} = \cC^\infty_{\rm odd}(\IS^2) :=\{u\in C^\infty(\IS^2)\,:\, u(-x)=-u(x),\;\forall x\in\IS^2\},
\] 
whereas
\[
\tilde{\cI}(C^\infty(\IS^2))=\cC^\infty_{\rm even}(\IS^2):=\{u\in C^\infty(\IS^2)\,:\, u(-x)=u(x),\;\forall x\in\IS^2\}.
\]
Therefore, 
\begin{equation}\label{e:ibij}
\tilde{\cI}:\cC^\infty_{\rm even}(\IS^2)\longrightarrow \cC^\infty_{\rm even}(\IS^2)\text{ is bijective.}
\end{equation}
Analogously, the Hamiltonian vector field $X_{\cI(V)}$ can be identified to a vector field on $\IS^2$ that is Hamiltonian with respect to the new symplectic form. In particular, its flow $\tilde{\phi}^V_s:=\phi^V_s\circ\Phi^{-1}$ satisfies
\begin{equation}\label{e:level}
\tilde{\phi}^V_s:\IS^2\To\IS^2,\quad \text{ and }\quad \tilde{\cI}(V)\circ\tilde{\phi}^V_s=\tilde{\cI}(V),\quad \forall s\in \IR.
\end{equation}
The strongest obstruction to $V$-GCC$_T$ comes from the fact that  $\cI(V)$ always has critical points:
\[
\cC(V)=\{\gamma \in G(\IS^2)\,:\, d\cI(V)_\gamma=0\}=\Phi^{-1}(\{p\in\IS^2\,:\, d\tilde{\cI}(V)_p=0\})\neq\emptyset.
\]
If $\gamma_0\in \cC(V)$ then $\phi^V_s(\gamma_0)=\gamma_0$ for every $s\in\IR$. Therefore, 
if $\omega$ satisfies $V$-GCC$_T$ then it must necessarily intersect the projection of $\gamma_0$ onto $\IS^2$.

Let us now define the class of potentials $\cT$. Let:
\[
Q_{(a,b,c)}(x)=ax^2_1+bx_2^2+cx_3^2,\quad x\in\IR^3.
\]
Then $Q_{(a,b,c)}|_{\IS^2}\in \cC^\infty_{\mathrm{even}}(\IS^2;\IR)$ and we define, using \eqref{e:ibij},
\[
\cT:=\tilde{\cI}^{-1}(\{Q_{(a,b,c)}|_{\IS^2}\,:\, 0<a<b<c\}).
\]
For any $V\in\cT$, the function $\tilde{\cI}(V)=Q_{(a,b,c)}|_{\IS^2}\in \cC^\infty_{\mathrm{even}}(\IS^2;\IR)$ has exactly six (non-degenerate) critical points:
\[
\{c_1^\pm:=(\pm 1,0,0),\;c_2^\pm:=(0,\pm 1, 0),\; c_3^\pm:=(0,0,\pm 1)\},
\]
and
\[
\min_{\IS^2}\tilde{\cI}(V)=a,\quad \max _{\IS^2}\tilde{\cI}(V)=c.
\]
The orbits of $\tilde{\phi}^V_t$ are contained in the connected components of the level sets $\tilde{\cI}(V)^{-1}(E)$, $E\in [a,c]$, by \eqref{e:level}. The set of orbits is invariant by the symmetries $x_i\mapsto -x_i$, for $i=1,2,3$; the precise description of the orbits is:
\begin{itemize}
\item The equilibrium points $c_1^\pm$ when $E=a$.
\item Two closed orbits around $c_1^\pm$ when $E\in (a,b)$.
\item The equilibrium points $c_2^\pm$ and four orbits connecting $c_2^+$ to $c_2^-$, when $E=b$.
\item Two closed orbits around $c_3^\pm$ when $E\in (b,c)$.
\item The equilibrium points $c_3^\pm$ when $E=c$.
\end{itemize}

The six oriented geodesics mapped by $\Phi$ to the critical points $c_i^\pm$ of $\tilde{\cI}(V)$ are $\gamma_1^\pm,\gamma_2^\pm,\gamma_3^\pm$, which correspond to the non-oriented geodesics:
\[
\gamma_i=\{x_i=0\},\quad i=1,2,3.
\]
Suppose that $\omega\subseteq\IS^2$ is an open set that contains $p=(0,0,1)\in\gamma_1^+\cap\gamma_2^+$ and $q=(0,1,0)\in\gamma_1^+\cap\gamma_3^+$. This means that the set of all geodesics that intersect $\omega$ contains $G_p\cup G_q$ (recall that this denotes the union of all geodesics issued from $p,q$). Now, $\Phi$ maps:
\[
\Phi(G_p)=\IS^2\cap \{x_3=0\}\quad\text{ and }\quad \Phi(G_q)=\IS^2\cap \{x_2=0\}.
\] 
The set $\Phi(G_p)$ is parameterized by $(\cos t,\sin t,0)$, and $\Phi(G_q)$ by $(\cos t,0,\sin t)$; evaluating along $Q_{(a,b,c)}$  shows that 
\[
(\Phi(G_p)\cup \Phi(G_q))\cap Q_{(a,b,c)}^{-1}(E)\neq \emptyset, \quad \forall E\in [a,c], 
\]
and, moreover,
\[
\quad c_2^\pm\in \Phi(G_p),\quad \Phi(G_q)\cap (Q_{(a,b,c)}^{-1}(b)\setminus\{c_2^\pm\})\neq \emptyset.  
\]
Since the set of orbits of $\tilde{\phi}^V_t$ and both sets $\Phi(G_p),\Phi(G_q)$ are invariant by the symmetries $x_i\mapsto -x_i$, for $i=1,2,3$, we conclude that $\Phi(G_p)\cup \Phi(G_q)$ has non-empty intersection with all the orbits of $\tilde{\phi}^V_t$. Therefore, $V$-GCC$_T$ is satisfied for some $T>0$ and the result follows from Corollary \ref{c:eig}.

Let us mention that six is the least number of critical points an even Morse function on $\IS^2$ may have. This is due to the fact that any such function induces a Morse function on the projective plane $\IP$. Since the Euler characteristic of $\IP$ is equal to one, the Poincaré-Hopf theorem implies:
\[
1=\chi(\IP)=\sum_{j=0}^2 (-1)^j\#\{\gamma\in\cC(V)\,:\,\gamma\text{ has index }j\}.
\]
There are at least one critical point of index zero and one of index two, therefore one must have also at least one saddle point. The number of critical points of $\cI(V)$ when viewed as a function of $\IP$ must be at least three, hence the claim.

\appendix

\section{Pseudo-differential operators and semiclassical measures}

\label{a:pdo} Here we review basic facts on the theory of semiclassical pseudo-differential operators and semiclassical measures that are used throughout the article. We refer, for instance, to \cite{DimassiSjostrand, F14, MaciaLille, Zwobook} for proofs and additional related materials. 

Let $(M,g)$ be a smooth Riemannian $d$-dimensional manifold without boundary.
Fix an atlas $(\varphi_l,U_l)$ of $M$, where each $\varphi$ is a
smooth diffeomorphism from $U_l\subset M$ onto its image $V_l$, an open set of $\IR^{d}$. We denote by 
$\varphi_l^*:\cC^{\infty}(V_l)\To \cC^{\infty}(U_l)$ the induced pull-back operators and by 
\[
\tilde{\varphi}_l \,:\,T^*U_l \To T^*V_l\,:\,(x,\xi)\longmapsto\left(\varphi_l(x),(d(\varphi_l)_x^{-1})^T\xi\right),
\]
the induced canonical transformation.
Consider now a smooth locally finite partition of unity
$(\phi_l)$ satisfying
$\sum_l\phi_l=1$ and $\phi_l\in \ml{C}_c^{\infty}(U_l)$. Then, any
function $a\in\ml{C}^{\infty}(T^*M)$ can be decomposed as 
$a=\sum_l a_l$, where $a_l:=\phi_l a$. Write
$\tilde{a}_l:=(\tilde{\varphi}_l^{-1})^*a_l\in \ml{C}_c^{\infty}(T^*V_l)$. Define the class of symbols of order $m$, depending on a small parameter $h\in(0,h_0]$:
\begin{equation}
\label{e:defpdo}S^{m}(T^{*}M):=\left\{a\in
\ml{C}^{\infty}(T^*M\times (0,h_0]):\sup_{(x,\xi),h,l}|\langle\xi\rangle^{|\beta|-m}\partial^{\alpha}_x\partial^{\beta}_{\xi}\tilde{a}_{l}(x,\xi,h)|\leq
C_{\alpha,\beta}\right\}.
\end{equation}
Given $a\in S^{m}(T^{*}M)$ and $l$, one defines the Weyl semiclassical pseudo-differential operator
$$\Op_{h}(\tilde{a}_l)u(x):=
\int_{\IR^{2d}}e^{i\xi\cdot(x-y)}\tilde{a}_l\left(\frac{x+y}{2},h\xi,h\right)u(y)dy\frac{d\xi}{(2\pi)^d},\quad \forall u\in\mathcal{S}(\mathbb{R}^d).$$
Finally, take $\psi_l\in \ml{C}_c^{\infty}(U_l)$ such
that $\psi_l=1$ close to the support of $\phi_l$. With these tools we define a Weyl semiclassical pseudo-differential operator of symbol $a\in S^{m}(T^*M)$ as follows:
\begin{equation}
\label{e:pdomanifold}\Op_{h}(a)(u):=\sum_l
\psi_l\left(\varphi_l^*\Op_{h}(\tilde{a}_l)(\varphi_l^{-1})^*\right)\left(\psi_l
u\right),\quad \forall u\in \ml{C}^{\infty}(M).
\end{equation}
It can be shown that the dependence
on the cutoffs and the atlas used to define this operator only appears at order $1$ in
$h$ (Theorem $9.10$ in~\cite{Zwobook}). 

The following properties are used several times in this article:
\begin{itemize}
\item Self-adjointness. If $a\in S^{m}(T^{*}M)$ is real-valued then $\Op_h(a)$ is self-adjoint on $L^2(M)$:
\begin{equation}\label{e:sa}
\Op_h(a)^*=\Op_h(a).
\end{equation}
\item Calderón-Vaillancourt theorem. There exist $C>0$ and $N_{d}\in\IN$ such that for every $a\in S^{0}(T^{*}M)$ one has
\begin{equation}\label{e:cv}
\|\Op_h(a)\|_{\cL(L^2(M))}\leq C  \sum_{\alpha,\beta\in\IN^d,\,|\alpha|+|\beta|\leq N_d} h^{\frac{|\alpha|+|\beta|}{2}}\sup_{T^*M}|\partial^\alpha_x\partial^\beta_\xi a|.
\end{equation}
\item Product rule. Let $a\in S^{m_1}(T^{*}M)$ and $b\in S^{m_2}(T^{*}M)$. Then there exists $r\in S^{m_1+m_2-1}(T^{*}M)$ such that:
\begin{equation}\label{e:prod}
\Op_h(a)\Op_h(b)=\Op_h(ab)+h\Op_h(r).
\end{equation}
\item Commutators. Let $a\in S^{m_1}(T^{*}M)$ and $b\in S^{m_2}(T^{*}M)$. Then there exists $r\in S^{m_1+m_2-2}(T^{*}M)$ such that:
\begin{equation}\label{e:comm}
[\Op_h(a),\Op_h(b)]=\frac{h}{i}\Op_h(\{a,b\})+h^2\Op_h(r),
\end{equation}
where $\{a,b\}$ is the canonical Poisson bracket of $a$ an $b$ in $T^*M$.
\item Functional calculus. Suppose $M$ is compact, then for every $\chi\in \cC^\infty_c(\IR)$ there exist $r\in S^{0}(T^{*}M)$ such that
\begin{equation}\label{e:fc}
\chi(-h^2\Delta)=\Op_h(\chi(\|\xi\|_x^2))+h\Op_h(r).
\end{equation}
\end{itemize}

The (semiclassical) Wigner distribution of a function $u\in L^2(M)$, which we denote by $W_u^h$ is the element of $\cD'(T^*M)$ defined by its action on test functions $a\in\cC^\infty_c(T^*M)$ by:
\begin{equation}\label{e:wigner}
\left\la W_u^h,a\right\ra:=(\Op_h(a)u\,|\,u)_{L^2(M)}.
\end{equation}
\begin{rem}\label{r:tbdd}
Suppose now that $(u_h)$ is a bounded sequence in $\cC(\IR;L^2(M))$ such that 
\begin{equation}\label{e:cmass}
\|u_h(t,\cdot)\|_{L^2(M)}=\|u_h(0,\cdot)\|_{L^2(M)}
\end{equation} 
for every $t\in \IR$. Then, by the Calderón-Vaillacourt theorem, the Wigner distributions $W^h_{u_h(t,\cdot)}$ are uniformly bounded in $\cD'(\IR\times M)$. 
\end{rem}
It can be shown (see \cite{MaciaAv, MacRiv16}), using the conservation hypothesis on the $L^2(M)$-norm \eqref{e:cmass} and G\aa{}rding's inequality that every accumulation point of the family $(W^h_{u_h})$ is an element of $L^\infty(\IR;\cM_+(T^*M))$, where $\cM_+(T^*M)$ stands for the cone of positive Radon measures on $T^*M$. Such accumulation points are called semiclassical measures of the family $(u_h)$; they are intrinsically defined on $\IR\times T^*M$, since, as we previously discussed, different definitions of $W_u^h$ obtained from the  atlas and partitions of unity used to define $\Op_h(a)$ differ by a $O(h)$ in $\cD'(\IR\times M)$.

Suppose $\mu$ is a semiclassical measure of $(u_h)$; this means that a sequence $h_n\to 0^+$ exists such that, for every $ \theta\in\cC^\infty_c(\IR)$ and $a\in\cC^\infty_c(T^*M)$,
\begin{equation}\label{e:tscm}
\int_\IR \theta(t)(\Op_{h_n}(a)u_{h_n}(t,\cdot)\,|\,u_{h_n}(t,\cdot))_{L^2(M)}dt\Tend{n}{\infty}\int_\IR\int_{T^*M}\theta(t)a(x,\xi)\mu_t(dx,d\xi)dt.
\end{equation}
If, in addition, the sequence satisfies the $h$-oscillation property:
\begin{equation}
\limsup_{h\to 0^+}\| \mathds{1}_{(R,\infty)}(-h^2\Delta)u_h(t,\cdot)\|\Tend{R}{\infty} 0,
\end{equation}
where $\mathds{1}_{(R,\infty)}$ stands for the characteristic function of the interval $(0,\infty)$, then, for every $b\in\cC_c(M)$,
\begin{equation}\label{e:proj}
\int_\IR \theta(t)b(x)|u_{h_n}(t,\cdot)|^2dx\,dt\Tend{n}{\infty}\int_\IR\int_{T^*M}\theta(t)b(x)\mu_t(dx,d\xi)dt;
\end{equation}
when $M$ is compact one has actually, for almost every $t\in\IR$,
\begin{equation}\label{e:prob}
\mu_t(T^*M)=\lim_{n\to\infty}\|u_{h_n}(t,\cdot)\|_{L^2(M)}^2.
\end{equation}

\def\cprime{$'$}


\begin{thebibliography}{10}

\bibitem{AFM15}
N.~Anantharaman, C.~Fermanian-Kammerer, and F.~Maci{\`a}.
\newblock Semiclassical completely integrable systems: long-time dynamics and
  observability via two-microlocal {W}igner measures.
\newblock {\em Amer. J. Math.}, 137(3):577--638, 2015.

\bibitem{ALMCras}
N.~Anantharaman, M.~L{\'e}autaud, and F.~Maci{\`a}.
\newblock Delocalization of quasimodes on the disk.
\newblock {\em C. R. Math. Acad. Sci. Paris}, 354(3):257--263, 2016.

\bibitem{ALM16}
N.~Anantharaman, M.~L{\'e}autaud, and F.~Maci{\`a}.
\newblock Wigner measures and observability for the {S}chr\"odinger equation on
  the disk.
\newblock {\em Invent. Math.}, 206(2):485--599, 2016.

\bibitem{AM14}
N.~Anantharaman and F.~Maci{\`a}.
\newblock Semiclassical measures for the {S}chr\"odinger equation on the torus.
\newblock {\em J. Eur. Math. Soc. (JEMS)}, 16(6):1253--1288, 2014.

\bibitem{ArM20}
V.~Arnaiz and F.~Maci{\`a}.
\newblock Concentration of Quasimodes for Perturbed Harmonic Oscillators.
\newblock {\em Preprint},  2020.

\bibitem{BBZ14}
J.~Bourgain, N.~Burq, and M.~Zworski.
\newblock Control for {S}chr\"odinger operators on 2-tori: rough potentials.
\newblock {\em J. Eur. Math. Soc. (JEMS)}, 15(5):1597--1628, 2013.

\bibitem{BZ12}
N.~Burq and M.~Zworski.
\newblock Control for {S}chr\"{o}dinger operators on tori.
\newblock {\em Math. Res. Lett.}, 19(2):309--324, 2012.

\bibitem{DimassiSjostrand}
M.~Dimassi and J.~Sj{\"o}strand.
\newblock {\em Spectral asymptotics in the semi-classical limit}, volume 268 of
  {\em London Mathematical Society Lecture Note Series}.
\newblock Cambridge University Press, Cambridge, 1999.

\bibitem{F14}
C.~Fermanian-Kammerer.
\newblock Op\'erateurs pseudo-diff\'erentiels semi-classiques.
\newblock In {\em Chaos en m\'ecanique quantique}, pages 53--100. Ed. \'Ec.
  Polytech., Palaiseau, 2014.

\bibitem{Guillemin1976}
V.~Guillemin.
\newblock The {R}adon transform on {Z}oll surfaces.
\newblock {\em Advances in Mathematics}, 22(1):85--119, 1976.

\bibitem{Jaffard1990}
S.~Jaffard.
\newblock Contr\^ole interne exact des vibrations d'une plaque rectangulaire.
\newblock {\em Portugal. Math.}, 47(4):423--429, 1990.

\bibitem{Komornik1992}
V.~Komornik.
\newblock On the exact internal controllability of a {P}etrowsky system.
\newblock {\em Journal de Math\'{e}matiques Pures et Appliqu\'{e}es. Neuvi\`eme
  S\'{e}rie}, 71(4):331--342, 1992.

\bibitem{Lebeau1992}
G.~Lebeau.
\newblock Contr\^ole de l'\'equation de {S}chr\"odinger.
\newblock {\em J. Math. Pures Appl. (9)}, 71(3):267--291, 1992.

\bibitem{Lebeau1996}
G.~Lebeau.
\newblock \'{E}quation des ondes amorties.
\newblock In {\em Algebraic and geometric methods in mathematical physics
  ({K}aciveli, 1993)}, volume~19 of {\em Math. Phys. Stud.}, pages 73--109.
  Kluwer Acad. Publ., Dordrecht, 1996.

\bibitem{LionsBook}
J.-L. Lions.
\newblock {\em Contr\^{o}labilit\'{e} exacte, perturbations et stabilisation de
  syst\`emes distribu\'{e}s. {T}ome 1}, volume~8 of {\em Recherches en
  Math\'{e}matiques Appliqu\'{e}es [Research in Applied Mathematics]}.
\newblock Masson, Paris, 1988.
\newblock Contr\^{o}labilit\'{e} exacte. [Exact controllability], With
  appendices by E. Zuazua, C. Bardos, G. Lebeau and J. Rauch.

\bibitem{MaciaAv}
F.~Maci{\`a}.
\newblock Semiclassical measures and the {S}chr\"odinger flow on {R}iemannian
  manifolds.
\newblock {\em Nonlinearity}, 22(5):1003--1020, 2009.

\bibitem{MaciaDispersion}
F.~Maci\`{a}.
\newblock The {S}chr\"odinger flow on a compact manifold: High-frequency
  dynamics and dispersion.
\newblock In {\em Modern Aspects of the Theory of Partial Differential
  Equations}, volume 216 of {\em Oper. Theory Adv. Appl.}, pages 275--289.
  Springer, Basel, 2011.

\bibitem{MaciaLille}
F.~Maci{\`a}.
\newblock High-frequency dynamics for the {S}chr\"odinger equation, with
  applications to dispersion and observability.
\newblock In {\em Nonlinear optical and atomic systems}, volume 2146 of {\em
  Lecture Notes in Math.}, pages 275--335. Springer, Cham, 2015.

\bibitem{MacRiv16}
F.~Maci{\`a} and G.~Rivi{\`e}re.
\newblock Concentration and non-concentration for the {S}chr\"odinger evolution
  on {Z}oll manifolds.
\newblock {\em Comm. Math. Phys.}, 345(3):1019--1054, 2016.

\bibitem{MacRiv19}
F.~Maci\`a and G.~Rivi\`ere.
\newblock Observability and quantum limits for the {S}chr\"{o}dinger equation
  on {$\Bbb{S}^d$}.
\newblock In {\em Probabilistic methods in geometry, topology and spectral
  theory}, volume 739 of {\em Contemp. Math.}, pages 139--153. Amer. Math.
  Soc., Providence, RI, 2019.

\bibitem{MZ02}
F.~Maci\`a and E.~Zuazua.
\newblock On the lack of observability for wave equations: a {G}aussian beam
  approach.
\newblock {\em Asymptot. Anal.}, 32(1):1--26, 2002.

\bibitem{Mi12}
L.~Miller.
\newblock Resolvent conditions for the control of unitary groups and their approximations. 
\newblock {\em J. Spectr. Theory}, 2(1):1--55. 2012.

\bibitem{PTZ}
Y.~Privat, E.~Tr\'{e}lat, and E.~Zuazua.
\newblock Optimal observability of the multi-dimensional wave and
  {S}chr\"{o}dinger equations in quantum ergodic domains.
\newblock {\em J. Eur. Math. Soc. (JEMS)}, 18(5):1043--1111, 2016.

\bibitem{RalstonGB}
J.~Ralston.
\newblock Gaussian beams and the propagation of singularities.
\newblock In {\em Studies in partial differential equations}, volume~23 of {\em
  MAA Stud. Math.}, pages 206--248. Math. Assoc. America, Washington, DC, 1982.

\bibitem{RauchTaylor75}
J.~Rauch and M.~Taylor.
\newblock Decay of solutions to nondissipative hyperbolic systems on compact
  manifolds.
\newblock {\em Comm. Pure Appl. Math.}, 28(4):501--523, 1975.

\bibitem{WeinsteinZoll}
A.~Weinstein.
\newblock Asymptotics of eigenvalue clusters for the {L}aplacian plus a
  potential.
\newblock {\em Duke Math. J.}, 44(4):883--892, 1977.

\bibitem{Zwobook}
M.~Zworski.
\newblock {\em Semiclassical analysis}, volume 138 of {\em Graduate Studies in
  Mathematics}.
\newblock American Mathematical Society, Providence, RI, 2012.

\end{thebibliography}
\end{document}